\documentclass{article}
\usepackage{graphicx}
\usepackage[left=3cm, top=2cm, right=3cm, bottom=3cm]{geometry}
\usepackage{listings}
\usepackage{float}
\usepackage{amssymb,amsmath,mathtools}
\usepackage{enumerate}
\usepackage[numbers,square,sort]{natbib}
\usepackage{wrapfig}
\usepackage{framed}
\usepackage{tabstackengine}
\usepackage{braket}
\usepackage[nottoc]{tocbibind}\settocbibname{References}
\usepackage{mathrsfs}
\usepackage{verbatim}

\usepackage{indentfirst}
\usepackage{setspace}
\usepackage{bm}

\usepackage{hyperref}
\hypersetup{colorlinks,citecolor=black,linkcolor=black}

\usepackage[dvipsnames]{xcolor}

\usepackage{tikz}
\tikzset{every loop/.style={}}
\usetikzlibrary{arrows,shapes,decorations.markings,automata,backgrounds,petri,bending,calc,angles}
\usepackage{circuitikz}

\usepackage{fancyhdr}
\fancypagestyle{firststyle}
{
	
	\fancyfoot[l]{\footnotesize\textit{Date:} \today}
	\fancyfoot[c]{1}
	\fancyhead{}

}

\usepackage{tikz-cd} 
\tikzset{
    labl/.style={anchor=south, rotate=90, inner sep=.5mm}
}
\tikzcdset{row sep/normal=2.7em}
\tikzcdset{column sep/normal=3.6em}

\usepackage{enumitem}
\setlist[enumerate,1]{label=(\arabic*)}
\setlist[enumerate,2]{label={(\alph*)},ref={(\alph*)}}
\setlist[enumerate,3]{label={(\roman*)},ref={(\roman*)}}
\newlist{steplist}{enumerate}{1}
\setlist[steplist]{label={Step \arabic*:}, ref={Step \arabic*}}

\makeatletter
\newcommand*{\wackyenum}[1]{%
	\expandafter\@wackyenum\csname c@#1\endcsname%
}
\newcommand*{\@wackyenum}[1]{%
	$\ifcase#1\or(1')\or(2')\or(3a')\or(3b')\or(4')%
	\else\@ctrerr\fi$%
}
\AddEnumerateCounter{\wackyenum}{\@wackyenum}{(1')}
\makeatother

\usepackage[utf8]{inputenc}
\usepackage[english]{babel} 

\usepackage{amsthm}
\newtheorem{thm}{Theorem}[section]
\newtheorem{lem}[thm]{Lemma}
\newtheorem{prop}[thm]{Proposition}

\numberwithin{equation}{section}

\newtheorem{que}[thm]{Question}

\theoremstyle{definition}
\newtheorem{defn}[thm]{Definition} 
\newtheorem{exmp}[thm]{Example} 
\newtheorem{remk}[thm]{Remark}

\newcommand{\G}{\Gamma}
\newcommand{\La}{\Lambda}

\newcommand{\bfX}{\mathbf{X}}

\newcommand{\bsT}{\mathbf{T^*}}
\newcommand{\sT}{T^*}

\newcommand{\Aut}{\operatorname{Aut}}

\makeatletter
\newsavebox{\@brx}
\newcommand{\llangle}[1][]{\savebox{\@brx}{\(\m@th{#1\langle}\)}%
	\mathopen{\copy\@brx\kern-0.5\wd\@brx\usebox{\@brx}}}
\newcommand{\rrangle}[1][]{\savebox{\@brx}{\(\m@th{#1\rangle}\)}%
	\mathclose{\copy\@brx\kern-0.5\wd\@brx\usebox{\@brx}}}
\makeatother

\begin{document}
\begin{center}
{\LARGE\bf
Common overlattices in trees and trees with fins}\\
\bigskip
\bigskip
{\large Sam Shepherd\footnote{Department of Mathematics, University of Münster, Münster, Germany. }}
\end{center}

\begin{abstract}
	Bass and Kulkarni proved that any pair of free uniform lattices in the automorphism group of a tree have conjugates that both lie inside a third uniform lattice (which is not necessarily free).
	We show that this does not generalise to trees with fins.
	The construction of our counter-example involves working with a certain generalisation of the universal groups of Burger and Mozes.
\end{abstract}

\textbf{Keywords} trees, trees with fins, uniform lattices, universal groups
\medskip

\textbf{Mathematics Subject Classification} 20F65 (Primary) 20E08, 22D05 (Secondary)
\medskip

\textbf{Email Address} sam.shepherd@uni-muenster.de
\tableofcontents

\bigskip
\section{Introduction}

Bass and Kulkarni proved the following theorem about uniform lattices in the automorphism group of a tree.

\begin{thm}\label{thm:BassKulkarni}\cite[Theorem 4.7]{BassKulkarni90}
	Let $T$ be a uniform tree. There exists a uniform lattice $\La<\Aut(T)$ that contains a conjugate of every free uniform lattice in $\Aut(T)$.
\end{thm}

One consequence of this theorem is that for any uniform lattices $\G,\G'<\Aut(T)$, there is a conjugate $g\G g^{-1}$ which is commensurable with $\G'$, i.e. $g\G g^{-1}\cap\G'$ has finite index in both $g\G g^{-1}$ and $\G'$.
In turn this is equivalent to Leighton's Theorem \cite{Leighton82}, which states that any pair of finite graphs with a common universal cover have a common finite cover (see also \cite[Corollary 4.16]{BassKulkarni90}).
There have been a number of generalisations of Leighton's Theorem to settings such as decorated graphs and right-angled buildings \cite{Haglund06,Neumann10,Huang18,Woodhouse21b,BridsonShepherd22,Shepherd22,Woodhouse23,Shepherd24}, as well as results concerning settings where such a generalisation does not hold \cite{Wise96,BurgerMozes00,Forester24,DergachevaKlyachko23,DergachevaKlyachko25}.
Plus there is work of Touikan \cite{Touikan25} which compares the conjugations $g\G g^{-1}$ required for Leighton's Theorem with the conjugations required for Theorem \ref{thm:BassKulkarni}.
In a slightly different direction, the $\Aut(T)$-commensurators of uniform lattices $\G<\Aut(T)$ have been studied in \cite{Liu94,AvniLimNevo12,Radu20}.

This article is concerned with generalisations (or limitations thereof) of Theorem \ref{thm:BassKulkarni} itself, not of Leighton's Theorem.
The first thing to observe here is that the freeness assumption in Theorem \ref{thm:BassKulkarni} is necessary.
Indeed, for the $n$-regular tree $T=T_n$ (with $n\geq3$), there exist uniform lattices in $\Aut(T)$ with arbitrarily large vertex stabilisers \cite[Theorem 7.1]{BassKulkarni90}, hence there cannot exist a uniform lattice $\La<\Aut(T)$ that contains a conjugate of every uniform lattice in $\Aut(T)$.
In fact, one can obtain an even stronger negation of Theorem \ref{thm:BassKulkarni} as follows.
If $\G<\Aut(T_n)$ is a vertex transitive uniform lattice such that for each vertex $x\in VT$ the local action of $\G_x$ realises every permutation of the edges incident to $x$, then the $\G$-pointwise-stabiliser of the 5-neighbourhood of any edge in $T$ is trivial \cite[Theorem 1.4]{TrofimovWeiss95}.
Such lattices exist (see e.g. Example \ref{exmp:LaPsi} and Lemma \ref{lem:Psi}), and because of the above restriction on pointwise-stabilisers there exists such a lattice $\G$ that is maximal among all uniform lattices in $\Aut(T)$.
Pairing this lattice $\G$ with another uniform lattice $\G'<\Aut(T)$ that has larger vertex stabilisers than $\G$ (which exists by \cite{BassKulkarni90}, as discussed above) yields the following theorem.

\begin{thm}\label{thm:nregular}\cite{BassKulkarni90,TrofimovWeiss95}
	Let $T$ be the $n$-regular tree, with $n\geq3$. There exist uniform lattices $\G,\G'<\Aut(T)$ such that no uniform lattice $\La<\Aut(T)$ contains conjugates of both $\G$ and $\G'$.
\end{thm}

The main theorem of this article concerns what happens when one replaces the tree $T$ with a slightly more general space.
Specifically, we consider a \emph{tree with fins}, which, roughly speaking, is a CAT(0) square complex obtained from a tree by attaching ``fins'' to a given collection of subtrees (see Definition \ref{defn:graphfin} for details).
Leighton's Theorem is known to generalise to trees/graphs with fins \cite{Woodhouse21b,BridsonShepherd22}, and trees with fins have played a key role in quasi-isometric rigidity for certain graphs of free groups \cite{ShepherdWoodhouse22,CashenMacura11}.
However, Theorem \ref{thm:BassKulkarni} does not generalise to trees with fins:

\begin{thm}\label{thm:fintree}
		There exists a locally finite tree with fins $\bfX$ and free uniform lattices $\G,\G'<\Aut(\bfX)$ such that no uniform lattice $\La<\Aut(\bfX)$ contains conjugates of both $\G$ and $\G'$.
\end{thm}

We emphasize that the lattices $\G,\G'$ in Theorem \ref{thm:fintree} are free, which makes the result quite different from Theorem \ref{thm:nregular}. In particular, arguments involving maximal lattices (as we used above for Theorem \ref{thm:nregular}) are unlikely to work for Theorem \ref{thm:fintree}.

In order to prove Theorem \ref{thm:fintree} we develop a version of the universal groups of Burger and Mozes, we prove an analogous theorem in the universal groups context (Theorem \ref{thm:univ}), and then we export this to the trees with fins setting (Section \ref{sec:fins}).
Our version of universal groups is more general than that of Burger and Mozes \cite{BurgerMozes00b}, but it is a special case of the universal groups of Reid and Smith \cite{ReidSmith26}.
Roughly speaking, our version of universal group is a subgroup $U^{(l)}(F)$ of the automorphism group of an $n$-regular tree $T$, such that the edges incident to each vertex can only be permuted according to a certain permutation group $F<S_n$ (and this is defined with the aid of an edge labelling $l$ on $T$).
The distinction between our version and that of Burger and Mozes comes down to how the edge labelling $l$ is defined (see Definitions \ref{defn:taulegal} and \ref{defn:unigroup} and Remark \ref{remk:BurgerMozes}), and this results in our version satisfying a universal property in a more general setting.
For our version, any vertex transitive group $H<\Aut(T)$ can be embedded in a universal group $U^{(l)}(F)$, such that the local action of each vertex stabiliser of $H$ is precisely $F$ (Proposition \ref{prop:universal}); for the Burger--Mozes universal groups, this only works when $H$ is edge transitive \cite[Proposition 3.2.2]{BurgerMozes00b} (which corresponds to $F<S_n$ being transitive).
The way we define the edge labelling is also crucial for the construction of the lattices required for Theorem \ref{thm:univ}, and hence for Theorem \ref{thm:fintree}.
The universal groups of Reid and Smith satisfy a universal property in an even broader setting, namely where $H$ can be any subgroup of $\Aut(T)$ (and $T$ need not be regular).
We note that our version of universal groups was developed independently of Reid and Smith, and we use different notation to them (indeed our notation is closer to that of Burger and Mozes).
The universal groups of Burger and Mozes have been influential in the study of groups acting on trees and also for groups acting on other spaces, for example the study of lattices in products of trees \cite{BurgerMozes00}.
There have been various generalisations of universal groups in the literature, not just \cite{ReidSmith26} but also \cite{LeBoudec16,Smith17,Tornier23}.

The examples of universal groups we construct also allow us to prove the following theorem regarding automorphism groups of trees with fins and their uniform lattices.

\begin{thm}\label{thm:finquotients}
There exists a locally finite uniform tree with fins $\bfX$ with the following properties:
\begin{enumerate}
	\item There is no uniform lattice $\G<\Aut(\bfX)$ such that the map $\G\backslash\bfX\to\Aut(\bfX)\backslash\bfX$ is an isomorphism.
	\item Let $X\subset\bfX$ be the underlying tree. 
	For $x\in VX$, let $E_X(x)$ denote the collection of edges of $X$ that are incident to $x$.
	There is no uniform lattice $\G<\Aut(\bfX)$ such that, for every $x\in VX$, every permutation of $E_X(x)$ induced by an automorphism in $\Aut(\bfX)_x$ is also induced by an automorphism in $\G_x$.
\end{enumerate}
\end{thm}

Our final result concerns a setting where Theorem \ref{thm:BassKulkarni} does actually generalise, namely the universal groups $U^{(l)}(F)$ for which $l$ is a legal labelling (see Definition \ref{defn:taulegal}).
This is precisely the case where our definition of universal groups coincides with the original definition of Burger and Mozes.

\begin{thm}\label{thm:legaluniv}
For any legal labelling $l$ on a regular tree $T$, and any universal group $U^{(l)}(F)$, there exists a uniform lattice $\La<U^{(l)}(F)$ that contains a $U^{(l)}(F)$-conjugate of every free uniform lattice in $U^{(l)}(F)$.
\end{thm}

The above theorems demonstrate contrasting behaviours in different settings.
On the one hand we have the behaviour of Theorems \ref{thm:BassKulkarni} and \ref{thm:legaluniv} for trees and certain universal groups (and the latter can be converted into a tree with fins example via Proposition \ref{prop:unitofins}), while on the other hand we have the behaviour of Theorem \ref{thm:fintree} for certain trees with fins.
It is therefore natural to ask when these different behaviours occur.

\begin{que}
	For which locally finite trees with fins $\bfX$ is it true that any pair of free uniform lattices $\G,\G'<\Aut(\bfX)$ can be conjugated into a third uniform lattice $\La<\Aut(\bfX)$? And when does there exist a single $\La$ that works for all choices of $\G,\G'$?
	And what if we replace $\Aut(\bfX)$ with a universal group $U^{(l)}(F)$ (or more generally a universal group from \cite{ReidSmith26})?
\end{que}

We provide some brief preliminaries in Section \ref{sec:prelim} regarding graphs and automorphism groups of trees.
In Section \ref{sec:universal} we develop our version of universal groups, and we prove Theorem \ref{thm:univ} (a universal groups version of Theorem \ref{thm:fintree}) and Theorem \ref{thm:legaluniv}. In Section \ref{sec:fins} we introduce trees with fins, and we deduce Theorems \ref{thm:fintree} and \ref{thm:finquotients} from the work in Section \ref{sec:universal}.

\textbf{Acknowledgements:}\,
The author would like to thank Colin Reid for some helpful discussions and comments.
The author was supported by the CRC 1442 Geometry: Deformations and Rigidity, and by the Deutsche Forschungsgemeinschaft (DFG, German Research Foundation) under Germany's Excellence Strategy EXC 2044/2 - 390685587, Mathematics Münster: Dynamics-Geometry-Structure.

\bigskip
\section{Preliminaries}\label{sec:prelim}

\begin{defn}(Graphs)\label{defn:graph}\\
	A \emph{graph} $X$ is a 1-dimensional cell complex. 
	We write $VX$ for the set of vertices (0-cells) and $EX$ for the set of edges (oriented 1-cells).
	Each edge $e$ has an \emph{origin} $o(e)\in VX$ and a \emph{terminus} $t(e)\in VX$.
	Associated to $e$ is the \emph{inverse edge} $\bar{e}\in EX$ given by reversing the orientation, which satisfies $o(\bar{e})=t(e)$ and $t(\bar{e})=o(e)$.
	The pair $\{e,\bar{e}\}$ is called a \emph{geometric edge} (and corresponds to a 1-cell in $X$).
	We write $E(x)$ for the set of edges with origin $x$.
\end{defn}

\begin{defn}(Uniform lattices in $\Aut(T)$)\label{defn:unilattice}\\
	Let $T$ be a locally finite tree. We write $\Aut(T)$ for the group of automorphisms of $T$, which we regard as a locally compact group using the compact-open topology.
	Following \cite{BassKulkarni90}, we say that $\G<\Aut(T)$ is a \emph{uniform lattice} if $\G$ acts properly and cocompactly on $T$ (i.e. the action of $\G$ on $VT$ has finite stabilisers and finitely many orbits), and we say that $T$ is \emph{uniform} if $\Aut(T)$ contains a uniform lattice.
	Note that $T$ is uniform if and only if it covers a finite graph \cite[Corollary 4.10]{BassKulkarni90}.
\end{defn}

\bigskip
\section{Universal groups}\label{sec:universal}

In this section we develop our version of universal groups. Our version is more general than that of Burger and Mozes \cite{BurgerMozes00b}, but a special case of the version developed by Reid and Smith \cite{ReidSmith26}. All the definitions and results up to Proposition \ref{prop:local=F} have analogues in \cite{ReidSmith26} (with the exception of Proposition \ref{prop:local=F}\ref{it:unimodular}). The parts that come after Proposition \ref{prop:local=F} are new.

Throughout this section let $T=T_n$ be the $n$-regular tree, with $n\geq3$.

\begin{defn}($\tau$-legal labellings)\\\label{defn:taulegal}
	Let $F<S_n$. Let $\sqcup_{i\in I}\Omega_i$ be the partition of $\{1,2,\dots,n\}$ into $F$-orbits.
	Let $\tau$ be an involution of $I$ (with $\tau$ being the identity allowed).
	A map $l: ET\to\{1,\dots,n\}$ is called a \emph{$\tau$-legal labelling of $T$} if
	\begin{enumerate}
		\item\label{it:l_x} for all $x\in VT$, the map $l_x:E(x)\to\{1,\dots,n\}$ given by $e\mapsto l(e)$ is a bijection, and
		\item\label{it:bare} for all $e\in ET$, if $l(e)\in \Omega_i$ then $l(\bar{e})\in\Omega_{\tau(i)}$.
	\end{enumerate}
Further, if $\tau$ is trivial and $l(e)=l(\bar{e})$ for all $e\in ET$ then we say that $l$ is a \emph{legal labelling of $T$}.
\end{defn}

\begin{defn}(Universal groups)\\\label{defn:unigroup}
	The \emph{universal group of $T$ with respect to $l$ and $F$} is
	\begin{equation*}
		U^{(l)}(F)=\{g\in\Aut(T)\mid l_{gx}\circ g\circ l_x^{-1}\in F\text{ for all }x\in VT\}.
	\end{equation*}
In particular, for $x\in VT$ the stabiliser $U^{(l)}(F)_x$ admits a homomorphism $\sigma_{x,l}:U^{(l)}(F)_x\to F$ given by $g\mapsto l_x\circ g\circ l_x^{-1}$ (and we will see later in Proposition \ref{prop:local=F} that this homomorphism is surjective).
\end{defn}

\begin{remk}\label{remk:BurgerMozes}
The universal groups developed and studied by Burger and Mozes in \cite{BurgerMozes00b} are the same as in Definition \ref{defn:unigroup} but restricted to the case where $l$ is a legal labelling, rather than a $\tau$-legal labelling.
The universal group $U(H,r)$ from \cite[Section 7.1]{ReidSmith26} is equivalent to $U^{(l)}(F)$ from Definition \ref{defn:unigroup} with $H=F$ and $r=\tau$ any involution of $I$.
\end{remk}

The universal groups satisfy the following universal property for vertex transitive subgroups of $\Aut(T)$.
This extends the universal property of the Burger--Mozes universal groups \cite[Proposition 3.2.2]{BurgerMozes00b}, which only applied to edge transitive subgroups of $\Aut(T)$ (corresponding to $F<S_n$ being a transitive permutation group).
The analogous result in the work of Reid and Smith is \cite[Theorem 3.10]{ReidSmith26}.

\begin{prop}\label{prop:universal}
	Let $H<\Aut(T)$ be vertex transitive. Then $H$ is contained in some universal group $U^{(l)}(F)$ for some $F<S_n$ and some $\tau$-legal labelling $l$ of $T$.
	Moreover, we can choose $U^{(l)}(F)$ so that the map $\sigma_{x,l}:H_x\to F$ is surjective for every $x\in VT$.
\end{prop}
\begin{proof}
	Let $I$ denote the set of $H$-orbits in $ET$ and let $\tau:I\to I$ be the involution such that $e\in i\in I$ implies $\bar{e}\in\tau(i)$ for all $e\in ET$.
	Fix a vertex $x_0\in VT$ and a bijection $l_{x_0}:E(x_0)\to\{1,\dots,n\}$.
	The local action of the stabiliser $H_{x_0}$ defines a permutation group $F:=l_{x_0}\circ H_{x_0}|_{E(x_0)}\circ l_{x_0}^{-1}<S_n$.
	The $F$-orbits in $\{1,\dots,n\}$ correspond to the $H_{x_0}$-orbits in $E(x_0)$, which in turn correspond to the $H$-orbits in $ET$ (since $H$ is vertex transitive).
	So we can define a partition $\{1,\dots,n\}=\sqcup_{i\in I}\Omega_i$ into $F$-orbits such that $l_{x_0}(e)\in\Omega_i$ if and only if $e\in i\in I$.
	
	We now extend $l$ to the rest of $ET$.
	For each $x\in VT$ choose an element $h_x\in H$ such that $h_xx=x_0$ (and choose $h_{x_0}=1$).
	Set
	\begin{equation}\label{definel}
l|_{E(x)}=l_{x_0}\circ h_x|_{E(x)}.
	\end{equation} 
	This defines a labelling $l:ET\to\{1,\dots,n\}$.
	The labelling $l$ satisfies Definition \ref{defn:taulegal}\ref{it:l_x} as a direct consequence of (\ref{definel}).
	Moreover, it follows from (\ref{definel}) and the way we defined the partition $\sqcup_{i\in I}\Omega_i$ that, for any $e\in ET$, we have $l(e)\in\Omega_i$ if and only if $e\in i\in I$.
	Definition \ref{defn:taulegal}\ref{it:bare} then follows from the fact that $e\in i\in I$ implies $\bar{e}\in\tau(i)$.
	Hence $l$ is a $\tau$-legal labelling.
	
	We now show that $H< U^{(l)}(F)$.
	Let $h\in H$ and $x\in VT$, and put $y=hx$.
	Then 
	\begin{equation}\label{l_yhl_x}
		l_y\circ h\circ l_x^{-1}=l_{x_0}\circ h_y\circ h\circ h_x^{-1}\circ l_{x_0}^{-1}\in F,
	\end{equation}
	because $h_yhh_x^{-1}\in H_{x_0}$. Thus $H< U^{(l)}(F)$.
	Moreover, if we specialise (\ref{l_yhl_x}) to the case where $h\in H_x$ (so $y=x$), we see that $\sigma_{x,l}(h)=l_x\circ h\circ l_x^{-1}=l_{x_0}\circ h_xhh_x^{-1}\circ l_{x_0}^{-1}$.
	As $h$ ranges over every element of $H_x$, the conjugate $h_xhh_x^{-1}$ will range over every element of $H_{x_0}$, hence $\sigma_{x,l}(h)$ will range over every element of $F$ (by construction of $F$).
	Thus $\sigma_{x,l}:H_x\to F$ is surjective.
\end{proof}

We now collect some lemmas which explore the extent to which $U^{(l)}(F)$ depends on the labelling $l$.
For $(f_x)\in F^{VT}$ and $l: ET\to\{1,\dots,n\}$ a map, we write $(f_x)\circ l$ for the map $ET\to\{1,\dots,n\}$ defined by $e\mapsto f_x(l(e))$ whenever $e\in E(x)$ and $x\in VT$.

\begin{lem}\label{lem:Ul=Ul'}
Let $(f_x)\in F^{VT}$ and let $l$ be a $\tau$-legal labelling of $T$.
Then $l'=(f_x)\circ l$ is also a $\tau$-legal labelling of $T$, and $U^{(l)}(F)=U^{(l')}(F)$.
\end{lem}
\begin{proof}
The fact that $l'$ is a $\tau$-legal labelling follows from the fact that each $f_x$ is a permutation of $\{1,\dots,n\}$ that stabilises each set $\Omega_i$.
For $g\in \Aut(T)$ and $x\in VT$, the condition $l_{gx}\circ g\circ l_x^{-1}\in F$ is equivalent to the condition
\begin{equation*}
	f_{gx}\circ l_{gx}\circ g\circ l_x^{-1}\circ f_x^{-1}=l'_{gx}\circ g\circ (l'_x)^{-1}\in F,
\end{equation*}
hence $U^{(l)}(F)=U^{(l')}(F)$.
\end{proof}

\begin{lem}\label{lem:lcircg}
	Let $l$ be a $\tau$-legal labelling of $T$ and let $g\in\Aut(T)$. Then $g\in U^{(l)}(F)$ if and only if $l\circ g=(f_x)\circ l$ for some $(f_x)\in F^{VT}$. In this case $U^{(l)}(F)=U^{(l\circ g)}(F)$.
\end{lem}
\begin{proof}
	We have $g\in U^{(l)}(F)$ if and only if there exists $(f_x)\in F^{VT}$ such that $l_{gx}\circ g\circ l_x^{-1}=f_x$ for each $x\in VT$.
	In turn, this is equivalent to $l\circ g=(f_x)\circ l$.
	In this case we have $U^{(l)}(F)=U^{(l\circ g)}(F)$ by Lemma \ref{lem:Ul=Ul'}.
\end{proof}

\begin{lem}\label{lem:ll'}
	Let $l,l'$ be $\tau$-legal labellings of $T$, let $x_0,x_0'\in VT$, and let $f_0\in F$.
	Then there exists $g\in\Aut(T)$ and $(f_x)\in F^{VT}$ such that:
	\begin{enumerate}
		\item\label{it:x0x'0} $gx_0=x'_0$,
		\item\label{it:fx0} $f_{x_0}=f_0$, and
		\item\label{it:ll'} $l'\circ g=(f_x)\circ l$.
	\end{enumerate}
\end{lem}
\begin{proof}
	We construct $g$ and $(f_x)$ inductively.
	For each integer $k\geq 0$ we define $g$ on the $(k+1)$-ball about $x_0$ and we define $(f_x)$ for $x$ in the $k$-ball about $x_0$.
	We require properties \ref{it:x0x'0}--\ref{it:ll'} to be satisfied at each stage (where defined).
	
	For $k=0$, we start by defining $gx_0=x'_0$ and $f_{x_0}=f_0$.
	Then we define $g|_{E(x_0)}:E(x_0)\to E(x'_0)$ to be the bijection that satisfies $l'_{x'_0}\circ g|_{E(x_0)}=f_{x_0}\circ l_{x_0}$.
	This uniquely determines $g$ on the 1-ball about $x_0$ (and it defines an isomorphism from the 1-ball about $x_0$ to the 1-ball about $x'_0$).
	
	Now let $k\geq1$, and suppose we have defined $g$ on the $k$-ball about $x_0$ and defined $(f_x)$ for $x$ in the $(k-1)$-ball about $x_0$.
	Let $x\in VT$ be at distance $k$ from $x_0$ and say $gx=x'$.
	Let $e\in E(x)$ be the edge pointing towards $x_0$, let $y=t(e)$ and let $ge=e'$ and $gy=y'$.
	We have $l'_{y'}\circ g|_{E(y)}=f_y\circ l_y$, so $l'(\bar{e}')=f_y(l(\bar{e}))$; thus $l'(\bar{e}'),l(\bar{e})$ are in the same $F$-orbit, say $\Omega_i$.
	As $l,l'$ are both $\tau$-legal labellings, we then have $l'(e'),l(e)\in\Omega_{\tau(i)}$.
	Let $f_x\in F$ be some element with $f_x(l(e))=l'(e')$.
	Then define $g|_{E(x)}:E(x)\to E(x')$ to be the bijection that satisfies $l'_{x'}\circ g|_{E(x)}=f_x\circ l_x$.
	Note that we still have $ge=e'$ since $f_x(l(e))=l'(e')$.
	Doing this for all $x\in VT$ at distance $k$ from $x_0$, we extend the domain of definition of $g$ from the $k$-ball about $x_0$ to the $(k+1)$-ball about $x_0$ (and it defines an isomorphism from the $(k+1)$-ball about $x_0$ to the $(k+1)$-ball about $x'_0$). This also extends the domain of definition of $(f_x)$ from the $(k-1)$-ball about $x_0$ to the $k$-ball about $x_0$. Moreover, properties \ref{it:x0x'0}--\ref{it:ll'} are satisfied (where defined).
\end{proof}

We can now show that $U^{(l)}(F)$ only depends on $F$ and $\tau$ up to conjugacy in $\Aut(T)$. This is consistent with the classical Burger--Mozes universal groups, where the dependence is only on $F$ -- see \cite[\S3.2]{BurgerMozes00b} or \cite[Corollary 4.4]{GarridoGlasnerTornier18}. The analogous result in the work of Reid and Smith is \cite[Theorem 3.12]{ReidSmith26}.

\begin{prop}\label{prop:UFtau}
Let $l,l'$ be $\tau$-legal labellings of $T$. Then $U^{(l)}(F)$ and $U^{(l')}(F)$ are conjugate in $\Aut(T)$.
Hence $U^{(l)}(F)$ only depends on $F$ and $\tau$ up to conjugacy.
\end{prop}
\begin{proof}
	By Lemma \ref{lem:ll'}, there are $g\in\Aut(T)$ and $(f_x)\in F^{VT}$ such that $l'\circ g=(f_x)\circ l$.
	Clearly $U^{(l'\circ g)}(F)=g^{-1}U^{(l')}(F)g$, and this group is equal to $U^{(l)}(F)$ by Lemma \ref{lem:Ul=Ul'}.
\end{proof}

We also collect the following consequences of Lemma \ref{lem:ll'}. Again this is consistent with the classical Burger--Mozes universal groups \cite{BurgerMozes00b,GarridoGlasnerTornier18}.
The analogous result in the work of Reid and Smith is \cite[Theorem 3.9]{ReidSmith26}.

\begin{prop}\label{prop:local=F}
	Let $l$ be a $\tau$-legal labelling of $T$.
	\begin{enumerate}
		\item\label{it:vertextrans} $U^{(l)}(F)$ acts vertex transitively on $T$.
		\item\label{it:local=F} For $x\in VT$, the homomorphism $\sigma_{x,l}:U^{(l)}(F)_x\to F$ (from Definition \ref{defn:unigroup}) is surjective.
		\item\label{it:unimodular} We have $|\Omega_i|=|\Omega_{\tau(i)}|$ for all $i\in I$ if and only if $U^{(l)}(F)$ contains a uniform lattice.
	\end{enumerate}	
\end{prop}
\begin{proof}
	\begin{enumerate}
		\item Apply Lemma \ref{lem:ll'} with $l=l'$ and $x_0,x_0'$ arbitrary vertices in $T$. This yields $g\in\Aut(T)$ with $gx_0=x_0'$ and $l\circ g=(f_x)\circ l$ for some $(f_x)\in F^{VT}$. Then by Lemma \ref{lem:lcircg} we have $g\in U^{(l)}(F)$.
		\item Apply Lemma \ref{lem:ll'} with $l=l'$ and $x_0=x'_0=x$ and let $f_0$ range over all elements of $F$.
		\item Suppose $|\Omega_i|=|\Omega_{\tau(i)}|$ for all $i\in I$.
		Let $X$ be the graph formed from two vertices $x_1,x_2$ and $n$ geometric edges, each joining $x_1$ with $x_2$.
		As $|\Omega_i|=|\Omega_{\tau(i)}|$ for all $i\in I$, one can endow $X$ with a $\tau$-legal labelling (noting that Definition \ref{defn:taulegal} makes sense for $n$-regular graphs as well as $n$-regular trees). Taking a covering map $T\to X$, we can then lift the labelling on $X$ to a $\tau$-legal labelling $l'$ on $T$. The deck group of the covering $T\to X$ will then be a uniform lattice in $U^{(l')}(F)$. By Proposition \ref{prop:UFtau}, $U^{(l')}(F)$ is conjugate to $U^{(l)}(F)$ in $\Aut(T)$, hence $U^{(l)}(F)$ contains a uniform lattice too.
		
		Now suppose that $U^{(l)}(F)$ contains a uniform lattice. By construction, $U^{(l)}(F)$ is a closed subgroup of $\Aut(T)$, so the existence of a uniform lattice in $U^{(l)}(F)$ implies that $U^{(l)}(F)$ is unimodular and that it admits a bi-invariant Haar measure $\mu$ (see \cite[Section 3 and Proposition 4.5]{BassKulkarni90} for the definition of unimodularity and its relationship with uniform lattices).
		Let $e\in ET$ with $l(e)\in\Omega_i$.
		By part \ref{it:local=F}, the orbit $U^{(l)}(F)_{o(e)}\cdot e$ is in bijection with $\Omega_i$ and the orbit $U^{(l)}(F)_{t(e)}\cdot e$ is in bijection with $\Omega_{\tau(i)}$.
		By part \ref{it:vertextrans}, the vertex stabilisers $U^{(l)}(F)_{o(e)}$ and $U^{(l)}(F)_{t(e)}$ are conjugate in $U^{(l)}(F)$, hence they have the same measure under $\mu$. Therefore 
		\begin{equation*}
			|\Omega_i|=|U^{(l)}(F)_{o(e)}:U^{(l)}(F)_e|=\frac{\mu(U^{(l)}(F)_{o(e)})}{\mu(U^{(l)}(F)_e)}=\frac{\mu(U^{(l)}(F)_{t(e)})}{\mu(U^{(l)}(F)_e)}=|U^{(l)}(F)_{t(e)}:U^{(l)}(F)_e|=|\Omega_{\tau(i)}|,
		\end{equation*} 
		as required.\qedhere
	\end{enumerate}	
\end{proof}

\begin{remk}
	The condition \ref{it:bare} from Definition \ref{defn:taulegal} -- namely that if $l(e)\in\Omega_i$ then $l(\bar{e})\in\Omega_{\tau(i)}$ -- is essential in the proof of Lemma \ref{lem:ll'}. In fact Lemma \ref{lem:ll'} and Propositions \ref{prop:UFtau} and \ref{prop:local=F} are false in general without this condition.
	Indeed, if a labelling $l$ satisfied Definition \ref{defn:taulegal}\ref{it:l_x} everywhere and satisfied Definition \ref{defn:taulegal}\ref{it:bare} for all edges apart from a single pair $e,\bar{e}$, then every automorphism in $U^{(l)}(F)$ would have to stabilise the geometric edge $\{e,\bar{e}\}$. From this one can construct counter-examples to Lemma \ref{lem:ll'} and Propositions \ref{prop:UFtau} and \ref{prop:local=F}.
\end{remk}

We now turn to the construction of the examples that will lead to Theorems \ref{thm:fintree} and \ref{thm:finquotients} from the introduction. We begin with the following lemma.

\begin{lem}\label{lem:compfactors}
	Let $\sqcup_{i\in I}\Omega_i=\Omega_1\sqcup\Omega_2\sqcup\dots\sqcup\Omega_m$.
	Suppose $F=F_1\times F_2$, where $F_1$ (resp. $F_2$) acts simply transitively on $\Omega_1$ (resp. $\Omega_2$) and trivially on $\Omega_2$ (resp. $\Omega_1$). Suppose $\tau(1)=2$.
	(The action of $F$ on the other $F$-orbits $\Omega_3,\dots,\Omega_m$ can be anything, and the restriction of $\tau$ to $\{3,\dots,m\}$ can be any involution.)
	Fix a $\tau$-legal labelling $l$ of $T$.
	Let $\G<U^{(l)}(F)$ be a uniform lattice.
	\begin{enumerate}
		\item\label{it:Omega2trans} Suppose that $\G$ acts vertex transitively on $T$, and suppose for each $x\in VT$ that $\G_x$ acts transitively on the part of $E(x)$ labelled by $\Omega_1$. Then for each $x\in VT$ we have that $\G_x$ also acts transitively on the part of $E(x)$ labelled by $\Omega_2$.
		\item\label{it:compfactors} Suppose that for each $x\in VT$ the stabiliser $\G_x$ acts transitively on the part of $E(x)$ labelled by $\Omega_1$ and on the part of $E(x)$ labelled by $\Omega_2$. Then the groups $F_1$ and $F_2$ have the same composition factors (with the same multiplicities).
	\end{enumerate}
\end{lem}
\begin{proof}
	\begin{enumerate}
		\item The assumptions on $\G$ mean that $\G$ acts transitively on the part of $ET$ labelled by $\Omega_1$. But $\tau(1)=2$, so the edges with $\Omega_1$ labels are precisely the inverses of the edges with $\Omega_2$ labels, hence $\G$ also acts transitively on the part of $ET$ labelled by $\Omega_2$.
		It follows that for each $x\in VT$ the stabiliser $\G_x$ acts transitively on the part of $E(x)$ labelled by $\Omega_2$.
		
		\item Let $e_1,e_2,e_3,\dots$ be the edges of an infinite oriented embedded path in $T$ such that $l(e_i)\in\Omega_1$ for all $i$. Say $o(e_i)=x_i$ and $t(e_i)=x_{i+1}$ for each $i$.
		We now consider the $\G$-stabiliser of one of the edges $e_i$.
		Since $\tau(1)=2$, we have $l(\bar{e}_i)\in\Omega_2$.
		The map $\sigma_{x_i,l}:U^{(l)}(F)_{x_i}\to F$ from Definition \ref{defn:unigroup} restricts to a map $\G_{x_i}\to F=F_1\times F_2$.
		The assumption that $\G_{x_i}$ acts transitively on the part of $E(x_i)$ labelled by $\Omega_1$ means that $\G_{x_i}$ surjects to the $F_1$ factor.
		Because $F_1$ (reps. $F_2$) acts simply transitively (resp. trivially) on $\Omega_1$, and since $l(e_i)\in\Omega_1$, we see that $\G_{e_i}$ is the kernel of the map $\G_{x_i}\to F_1$.
		A similar argument shows that $\G_{e_i}=\G_{\bar{e}_i}$ is the kernel of a similar map $\G_{x_{i+1}}\to F_2$.
		Let $\mathfrak{C}(G)$ denote the multiset of composition factors of a finite group $G$ (considered as a multiset of isomorphism classes of groups).
		The above observations imply that
		\begin{equation}\label{compfactors}
\mathfrak{C}(\G_{x_i})=\mathfrak{C}(\G_{e_i})\sqcup\mathfrak{C}(F_1)\quad\text{and}\quad \mathfrak{C}(\G_{x_{i+1}})=\mathfrak{C}(\G_{e_i})\sqcup\mathfrak{C}(F_2).
		\end{equation} 
		Since there are finitely many $\G$-orbits of vertices in $T$, we must have $\G_{x_i}$ conjugate to $\G_{x_j}$ for some $i<j$. In particular, $\mathfrak{C}(\G_{x_i})=\mathfrak{C}(\G_{x_j})$. Applying (\ref{compfactors}) to the edges $e_i,e_{i+1},\dots,e_{j-1}$, we deduce that $\mathfrak{C}(F_1)=\mathfrak{C}(F_2)$, as required.\qedhere
	\end{enumerate}
\end{proof}

\begin{exmp}\label{exmp:fewerorbits}
	Let $n=120$. Say $\{1,\dots,120\}=\Omega_1\sqcup\Omega_2=\{1,\dots,60\}\sqcup\{61,\dots,120\}$.
	Let $F=A_5\times C_{60}$, with $A_5$ (resp. $C_{60}$) acting simply transitively on $\Omega_1$ (resp. $\Omega_2$) and trivially on $\Omega_2$ (resp. $\Omega_1$). Let $\tau=(12)\in S_2$, and let $l$ be a $\tau$-legal labelling of $T=T_{120}$.
	We have $|\Omega_1|=|\Omega_2|=60$, so $U^{(l)}(F)$ contains a uniform lattice by Proposition \ref{prop:local=F}\ref{it:unimodular}.
	On the other hand, the groups $A_5$ and $C_{60}$ have different composition factors (indeed $A_5$ is simple while $C_{60}$ has composition factors $C_2,C_2,C_3,C_5$), so	by Lemma \ref{lem:compfactors}\ref{it:compfactors} we see that there is no uniform lattice $\G< U^{(l)}(F)$ with the property that the map $\sigma_{x,l}:\G_x\to F$ (from Definition \ref{defn:unigroup}) is a surjection for every $x\in VT$.
	Additionally, Lemma \ref{lem:compfactors}\ref{it:compfactors} implies that no uniform lattice in $U^{(l)}(F)$ acts transitively on the geometric edges of $T$. Hence every uniform lattice in $U^{(l)}(F)$ has strictly more orbits of geometric edges than $U^{(l)}(F)$ itself.
\end{exmp}

\begin{exmp}\label{exmp:nooverlattice}
	Let $n=240$. Say $\{1,\dots,240\}=\Omega_1\sqcup\Omega_2\sqcup\Omega_3\sqcup\Omega_4$, with
	\begin{align*}
		\Omega_1&=\{1,\dots,60\}\\
		\Omega_2&=\{61,\dots,120\}\\
		\Omega_3&=\{121,\dots,180\}\\
		\Omega_4&=\{181,\dots,240\}.
	\end{align*}
	Let $F=A_5\times C_{60}$, with $A_5$ (resp. $C_{60}$) acting simply transitively on $\Omega_1,\Omega_3$ and $\Omega_4$ (resp. $\Omega_2$) and trivially on $\Omega_2$ (resp. $\Omega_1,\Omega_3$ and $\Omega_4$). Let $\tau=(12)(34)\in S_4$ and let $T=T_{240}$.
	Furthermore, suppose that the map $\Omega_3\to\Omega_4$ given by  +60 is $A_5$-equivariant.
	
	Let $X,X'$ be the labelled graphs shown in Figure \ref{fig:XX'}. Let $\rho:T\to X$ and $\rho':T\to X'$ be covering maps, and let $\G,\G'<\Aut(T)$ be the corresponding deck groups. The labellings on $X$ and $X'$ lift to $\tau$-legal labellings $l$ and $l'$ on $T$.
	Observe that $\G<U^{(l)}(F)$ and $\G'<U^{(l')}(F)$.
	
	We may modify $\rho'$ by composing it with an automorphism $g\in\Aut(T)$, hence also modifying $l'$ and $\G'$.
	By Lemma \ref{lem:ll'}, we may choose this modification so that $l'=(f_x)\circ l$ for some $(f_x)\in F^{VT}$.
	Then by Lemma \ref{lem:Ul=Ul'}, we see that $\G$ and $\G'$ are both uniform lattices in the same universal group $U^{(l)}(F)=U^{(l')}(F)$.
\end{exmp}

\begin{figure}[H]
		\begin{tikzpicture}[auto,node distance=2cm,
			thick,every node/.style={circle,draw,font=\small},
			every loop/.style={min distance=2cm},
			hull/.style={draw=none},
			triangle/.style = {regular polygon, regular polygon sides=3, rotate=-90,fill,scale=0.5},
			]
			\tikzstyle{label}=[draw=none,font=\normalsize]
		\node[fill] at (0,0) (X){};		
		
		\draw[fill=none](0,-1) circle (1.0){};
		\draw[fill=none](0,-1.4) circle (1.4){};
		\draw[fill=none](0,-2) circle (2.0){};
		\path (0,-3) edge [dashed] (0,-3.8);
		\node[label] at (-.9,-1.9){121};
		\node[label] at (.9,-1.9){181};
		\node[label] at (-1.5,-2.3){122};
		\node[label] at (1.5,-2.3){182};
		\node[label] at (-2.16,-2.84){180};
		\node[label] at (2.16,-2.84){240};
		
		\draw[fill=none](0,1) circle (1.0){};
		\draw[fill=none](0,1.4) circle (1.4){};
		\draw[fill=none](0,2) circle (2.0){};
		\path (0,3) edge [dashed] (0,3.8);
		\node[label] at (-.9,1.9){1};
		\node[label] at (.9,1.9){61};
		\node[label] at (-1.5,2.3){2};
		\node[label] at (1.5,2.3){62};
		\node[label] at (-2.16,2.84){60};
		\node[label] at (2.16,2.84){120};
		
		\node[label,font=\huge] at (0,-5) {$X$};
		
		\node[triangle] at (0,2){};
		\node[label,font=\large] at (0,1.7) {$e$};
		
		\begin{scope}[shift={(10,0)}]
			\node[fill] at (0,0) (X){};		
			
			\draw[fill=none](0,-1) circle (1.0){};
			\draw[fill=none](0,-1.4) circle (1.4){};
			\draw[fill=none](0,-2) circle (2.0){};
			\path (0,-3) edge [dashed] (0,-3.8);
			\node[label] at (-.9,-1.9){121};
			\node[label] at (.9,-1.9){182};
			\node[label] at (-1.5,-2.3){122};
			\node[label] at (1.5,-2.3){183};
			\node[label] at (-2.16,-2.84){180};
			\node[label] at (2.16,-2.84){181};
			
			\draw[fill=none](0,1) circle (1.0){};
			\draw[fill=none](0,1.4) circle (1.4){};
			\draw[fill=none](0,2) circle (2.0){};
			\path (0,3) edge [dashed] (0,3.8);
			\node[label] at (-.9,1.9){1};
			\node[label] at (.9,1.9){61};
			\node[label] at (-1.5,2.3){2};
			\node[label] at (1.5,2.3){62};
			\node[label] at (-2.16,2.84){60};
			\node[label] at (2.16,2.84){120};
			
			\node[label,font=\huge] at (0,-5) {$X'$};
		\end{scope}
		\end{tikzpicture}
	\caption{\small The labelled graphs $X$ and $X'$ used in Example \ref{exmp:nooverlattice}. Note that $X'$ is obtained from $X$ by cyclically permuting the labels $181,182,\dots,240$. Note that each label is shown near the origin end of the edge, for example the edge $e$ in $X$ has label 1 while the inverse edge $\bar{e}$ has label 61.}\label{fig:XX'}
\end{figure}
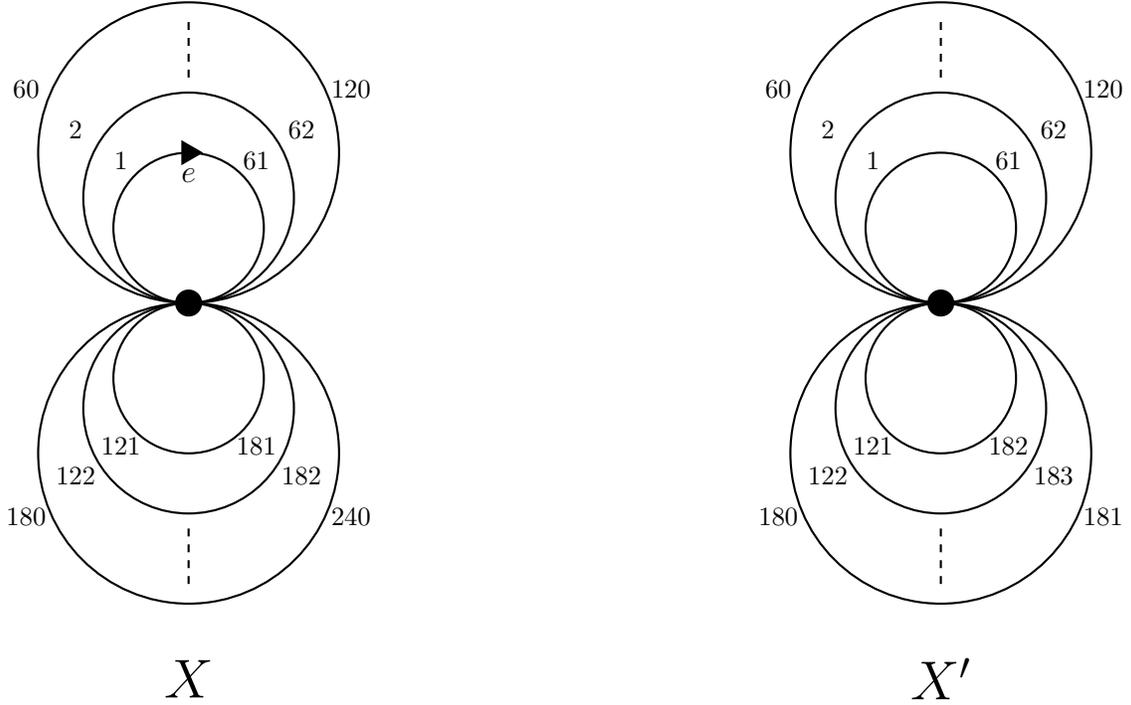

Using Example \ref{exmp:nooverlattice}, we now prove the following theorem, which will be used to deduce Theorem \ref{thm:fintree} in Section \ref{sec:fins}.

\begin{thm}\label{thm:univ}
	There exists a $\tau$-legal labelling $l$ on a tree $T$, a universal group $U^{(l)}(F)$, and free uniform lattices $\G,\G'<U^{(l)}(F)$ such that no uniform lattice $\La<U^{(l)}(F)$ contains $U^{(l)}(F)$-conjugates of both $\G$ and $\G'$.
\end{thm}
\begin{proof}
	Let $\G,\G'<U^{(l)}(F)=U^{(l')}(F)$ be the lattices constructed in Example \ref{exmp:nooverlattice}.
	It suffices to show that no uniform lattice $\La<U^{(l)}(F)$ contains both $\G$ and a conjugate of $\G'$.
	Suppose the contrary, that $\La<U^{(l)}(F)$ is a uniform lattice containing both $\G$ and $\G''=g^{-1}\G'g$, with $g\in U^{(l)}(F)$.
	
	Recall from Example \ref{exmp:nooverlattice} that $\G,\G'$ are the deck groups of covering maps $\rho:T\to X$ and $\rho':T\to X'$.
	It follows that $\G''$ is the deck group of the covering map $\rho''=\rho'\circ g:T\to X'$.
	Moreover, the labelling on $X'$ lifts via $\rho''$ to a labelling $l''=l'\circ g$ on $T$.
	By Lemma \ref{lem:lcircg}, we have $l''=(\bar{f}_x)\circ l'$ for some $(\bar{f}_x)\in F^{VT}$.
	But from Example \ref{exmp:nooverlattice} we also have $l'=(f_x)\circ l$ with $(f_x)\in F^{VT}$.
	Therefore $l=(f'_x)\circ l''$ with $(f'_x)=(f_x^{-1}\bar{f}_x^{-1})\in F^{VT}$.
	
	Pick a vertex $x\in VT$. Let $e_{121},e_{181},e_{182}\in E(x)$ have $l''$-labels 121, 181, 182 respectively. Since $l''$ is a lift of the labelling on $X'$, we have $l''(\bar{e}_{121})=182$. Hence there is a deck transformation $\gamma''\in\G''$ with $\gamma''(\bar{e}_{121})=e_{182}$ (see Figure \ref{fig:3edges}). 
	Now consider the $l$-labels of these edges.
	We have $l=(f'_x)\circ l''$, therefore the edges $e_{121},e_{181},e_{182}$ have $l$-labels $f'_x(121), f'_x(181), f'_x(182)$ respectively.
	Since $l$ is a lift of the labelling on $X$, we have $l(\bar{e}_{121})=f'_x(121)+60$. Moreover, we assumed in Example \ref{exmp:nooverlattice} that the map $\Omega_3\to\Omega_4$ given by +60 is $A_5$- and hence $F$-equivariant, so in fact we have $l(\bar{e}_{121})=f'_x(181)=l(e_{181})$.
	It follows that there is a deck transformation $\gamma\in\G$ with $\gamma(\bar{e}_{121})=e_{181}$ (again see Figure \ref{fig:3edges}).
	Then the composition $\lambda=\gamma''\gamma^{-1}\in\La$ satisfies $\lambda(e_{181})=e_{182}$ (and $\lambda(x)=x$).
	
	A similar argument to above shows that for each $i\in\{181,182,\dots,239\}$ there is an element of $\La$ that fixes $x$ and maps $e_i$ to $e_{i+1}$, where each $e_i$ is the edge in $E(x)$ with $l''(e_i)=i$.
	It follows that the stabiliser $\La_x$ acts transitively on the edges in $E(x)$ with $l''$-labels in $\Omega_4$.
	Recall that the $A_5$ factor in $F=A_5\times C_{60}$ acts simply transitively on both $\Omega_1$ and $\Omega_4$, while the $C_{60}$ factor acts trivially on $\Omega_1$ and $\Omega_4$.
	Since $\La<U^{(l)}(F)=U^{(l'')}(F)$ (Lemma \ref{lem:Ul=Ul'}), we deduce that the map $\sigma_{x,l''}:\La_x\to F$ (given by Definition \ref{defn:unigroup}) surjects to the $A_5$ factor of $F$. Therefore $\La_x$ acts transitively on the edges in $E(x)$ with $l''$-labels in $\Omega_1$.
	
	This whole argument works for any $x\in VT$, and moreover $\La$ acts vertex transitively on $T$ since it contains $\G$ and $\G'$.
	We can then apply Lemma \ref{lem:compfactors} to deduce that the factors of $F=A_5\times C_{60}$ have the same composition factors, a contradiction.
\end{proof}

\begin{figure}[H]
	\centering
	\begin{tikzpicture}[auto,node distance=2cm,
		thick,every node/.style={circle,draw,font=\small},
		every loop/.style={min distance=2cm},
		hull/.style={draw=none},
		]
		\tikzstyle{label}=[draw=none,font=\normalsize]
		\node[fill] at (0,0) (X){};
		
		\path (0,0) edge [black,postaction={decoration={markings,mark=at position 0.6 with {\arrow[black,line width=1mm, scale=0.7]{triangle 60}}},decorate}] (0,-3);
		\path (0,0) edge [black,postaction={decoration={markings,mark=at position 0.6 with {\arrow[black,line width=1mm, scale=0.7]{triangle 60}}},decorate}] ++(130:3);	
		\path (0,0) edge [black,postaction={decoration={markings,mark=at position 0.6 with {\arrow[black,line width=1mm, scale=0.7]{triangle 60}}},decorate}] ++(50:3);	
		
		\node[label,font=\large] at (.6,-1.6){$e_{121}$};
		\node[label,font=\large] at (1.6,1.1){$e_{182}$};
		\node[label,font=\large] at (-1.6,1.1){$e_{181}$};
		
		\node[label,font=\large] at (.4,0){$x$};
		
		\node[label,font=\normalsize] at (2.3,.6){$l=f'_x(182), l''=182$};
		\node[label,font=\normalsize] at (-2.2,.6){$l=f'_x(181), l''=181$};
		\node[label,font=\normalsize] at (1.7,-.8){$l=f'_x(121), l''=121$};
		\node[label,font=\normalsize] at (1.7,-2.3){$l=f'_x(181), l''=182$};
		
		\path (-2.6,-3) edge [red,bend right=20,postaction={decoration={markings,mark=at position 0.6 with {\arrow[red,line width=1mm, scale=0.7]{triangle 60}}},decorate}] (-4.6,1.2);	
		\path (3.3,-3) edge [red,bend left=20,postaction={decoration={markings,mark=at position 0.6 with {\arrow[red,line width=1mm, scale=0.7]{triangle 60}}},decorate}] (5.3,1.2);	
		\node[label,red,font=\large] at (-3.7,-.8) {$\gamma$};
		\node[label,red,font=\large] at (4.4,-.8) {$\gamma''$};
	
	\end{tikzpicture}
	\caption{\small The $l$ and $l''$-labellings on the edges $e_{121},\bar{e}_{121},e_{181},e_{182}$, as in the proof of Theorem \ref{thm:univ} (note that the labellings on $e_{121}$ are shown near the origin of the edge -- in this case near the vertex $x$ -- while the labellings on $\bar{e}_{121}$ are shown near the terminus of $e_{121}$).}\label{fig:3edges}
\end{figure}
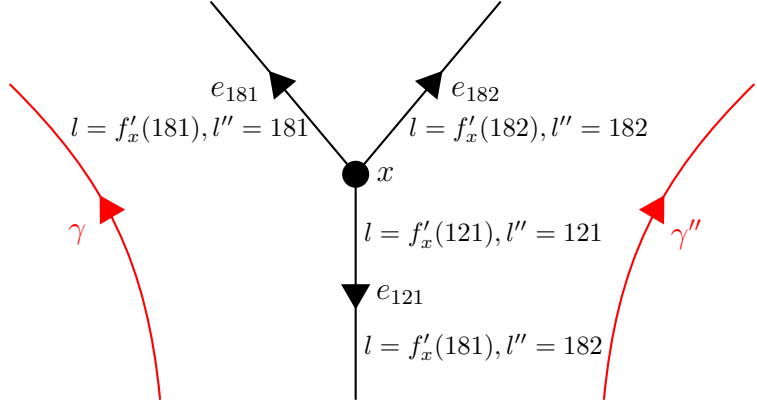

We now turn to Theorem \ref{thm:legaluniv}, regarding the case where $l$ is a legal labelling. This relies on the following example of a uniform lattice in $U^{(l)}(F)$.

\begin{exmp}\label{exmp:LaPsi}
	Let $n$ and $F<S_n$ be arbitrary. Let $l$ be a legal labelling on $T=T_n$ (recall this means that $l(e)=l(\bar{e})$ for all $e\in ET$). Let
	\begin{align}
		\La^l(F)&=\{g\in \Aut(T)\mid\exists f\in F,\,\forall e\in ET:\, l(ge)=fl(e)\},\label{falle}\\
		&=\{g\in \Aut(T)\mid\exists f\in F,\,\forall x\in VT:\, l_{gx}\circ g\circ l_x^{-1}=f\}.\label{fallx}
	\end{align}
It follows from (\ref{fallx}) that $\La^l(F)$ is a subgroup of $U^{(l)}(F)$.
And $\La^l(F)$ is vertex transitive since the group of automorphisms of $T$ that preserve $l$ is vertex transitive -- see \cite[\S3.2]{BurgerMozes00b} or \cite[Lemma 4.3]{GarridoGlasnerTornier18}.
Let $\Psi:\Lambda^l\to F$ be the map that associates to each $g\in\Lambda^l$ the element $f\in F$ from (\ref{fallx}) (or equivalently from (\ref{falle})).
It is easy to check that $\Psi$ is a homomorphism.
\end{exmp}

\begin{lem}\label{lem:Psi}
	The map $\Psi:\La^l(F)\to F$ from Example \ref{exmp:LaPsi} restricts to an isomorphism $\La^l(F)_x\to F$ for any $x\in VT$.
	In particular, $\La^l(F)$ is a uniform lattice in $U^{(l)}(F)$.
\end{lem}
\begin{proof}
	Let $f\in F$. We now construct $g\in\La^l(F)_x\cap\Psi^{-1}(f)$ inductively by defining it on the $k$-ball about $x$ for each $k=1,2,\dots$, so that at each step $l_{gx}\circ g\circ l_x^{-1}=f$ for all $x\in VT$ in the $(k-1)$-ball about $x$.
	For $k=1$, we define $g$ on the 1-ball about $x$ by setting $g|_{E(x)}=l_x^{-1}\circ f\circ l_x:E(x)\to E(x)$.
	Now suppose we have defined $g$ on the $k$-ball about $x$. Let $y\in VT$ be at distance $k$ from $x$, and let $e\in E(y)$ be the edge pointing towards $x$. By the induction hypothesis combined with the fact that $l$ is a legal labelling, $l(ge)=l(g\bar{e})=fl(\bar{e})=fl(e)$. We then extend $g$ to the rest of $E(y)$ by setting $g|_{E(y)}=l_{gy}^{-1}\circ f\circ l_y:E(y)\to E(gy)$. Doing this for all $y\in VT$ at distance $k$ from $x$ allows us to extend $g$ to the $(k+1)$-ball about $x$ in the required manner.
	
	Finally, note that the choices $g|_{E(y)}$ made in the above construction were the only choices that ensured $l_{gy}\circ g\circ l_y^{-1}=f$ for all $y\in VT$, so the element $g\in\La^l(F)_x\cap\Psi^{-1}(f)$ is unique.
\end{proof}

\begin{lem}\label{lem:Laconj}
	If $l$ is a legal labelling on $T$ and $h\in\Aut(T)$ then $\La^{l\circ h}(F)=h^{-1}\La^l(F)h$.
\end{lem}
\begin{proof}
	First note that $l\circ h$ is a legal labelling ($l(e)=l(\bar{e})$ for all $e\in ET$ clearly implies $(l\circ h)(e)=(l\circ h)(\bar{e})$ for all $e\in ET$), so the group $\La^{l\circ h}(F)$ is well-defined.
For $g\in\Aut(T)$ and $f\in F$, the condition $l(ge)=fl(e)$ for all $e\in ET$ from (\ref{falle}) is equivalent to $(l\circ h)(h^{-1}ghh^{-1}e)=f(l\circ h)(h^{-1}e)$ for all $e\in ET$, which in turn is equivalent to  $(l\circ h)(h^{-1}ghe)=f(l\circ h)(e)$ for all $e\in ET$. It follows that $g\in\La^l(F)$ if and only if $h^{-1}gh\in\La^{l\circ h}(F)$.
\end{proof}

\theoremstyle{plain}
\newtheorem*{thm:legaluniv}{Theorem \ref{thm:legaluniv}}
\begin{thm:legaluniv}
	For any legal labelling $l$ on a regular tree $T$, and any universal group $U^{(l)}(F)$, there exists a uniform lattice $\La<U^{(l)}(F)$ that contains a $U^{(l)}(F)$-conjugate of every free uniform lattice in $U^{(l)}(F)$.
\end{thm:legaluniv}
\begin{proof}
	Fix a legal labelling $l$ on an $n$-regular tree $T$, and fix $F<S_n$.
	Let $\La^l(F)< U^{(l)}(F)$ be the group from Example \ref{exmp:LaPsi}, which is a uniform lattice by Lemma \ref{lem:Psi}.
	We claim that $\La^l(F)$ contains a $U^{(l)}(F)$-conjugate of every free uniform lattice in $U^{(l)}(F)$.
	
	Let $\G<U^{(l)}(F)$ be a free uniform lattice. Since $\G$ acts freely on $T$, it is a free group. Moreover, as $\G$ acts cocompactly on $T$, there is a finite subtree $Y\subset T$ such that the $\G$-translates of $VY$ partition $VT$.
	Furthermore, there is a free basis $\{\gamma_1\dots,\gamma_r\}$ of $\G$ and vertices $x_1,\dots,x_r\in VY$ such that $\gamma_ix_i$ is adjacent to $Y$ for each $i$. Say $e_i\in E(\gamma_ix_i)$ is the edge going from $\gamma_ix_i$ to $Y$.
	Note that every edge in $T$ is either a $\G$-translate of an edge in $Y$ or a $\G$-translate of an edge $e_i$ or $\bar{e}_i$ (indeed these translates of edges form a $\G$-invariant subtree of $T$, which must equal the whole of $T$).
	
	As $\G<U^{(l)}(F)$, there are elements $f_1,\dots,f_r\in F$ such that $l_{\gamma_ix_i}\circ\gamma_i\circ l_{x_i}^{-1}=f_i$ for each $i$.	
	Let $\theta:\G\to F$ be the unique homomorphism with $\theta(\gamma_i)=f_i$ for each $i$.
	Define a labelling $l':ET\to\{1,\dots,n\}$ by setting 
	\begin{equation}\label{l'def}
l'_{h y}=\theta(h)\circ l_y\circ h^{-1}|_{E(h y)},
	\end{equation}
for each $h\in\G$ and $y\in VY$ (this is well-defined since for each $x\in VT$ there are unique $h\in\G$ and $y\in VY$ with $x=hy$).
	
	We claim that $l'$ is $\theta$-equivariant in the sense that $l'_{gx}\circ g|_{E(x)}=\theta(g)\circ l'_x$ for every $g\in\G$ and $x\in VT$ (or equivalently $l'(ge)=\theta(g)l'(e)$ for every $g\in \G$ and $e\in ET$).
	Indeed, given $x\in VT$, we can write $x=hy$ for $h\in \G$ and $y\in VY$.
	Then for $g\in\G$ we have
	\begin{align*}
l'_{gx}\circ g|_{E(x)}&=l'_{ghy}\circ g|_{E(hy)}\\
&=\theta(gh)\circ l_y\circ(gh)^{-1}\circ g|_{E(hy)}\\
&=\theta(g)\circ\theta(h)\circ l_y\circ h^{-1}|_{E(hy)}\\
&=\theta(g)\circ l'_x,
	\end{align*}
which proves the claim.
	
	Next, we claim that $l'$ is a legal labelling of $T$.
	By (\ref{l'def}) we know that each $l'_x:E(x)\to\{1,\dots,n\}$ is a bijection, so it suffices to show that $l'(e)=l'(\bar{e})$ for all $e\in ET$.
	As observed earlier, every edge in $T$ is either a $\G$-translate of an edge in $Y$ or a $\G$-translate of an edge $e_i$ or $\bar{e}_i$.
	By (\ref{l'def}), $l'$ agrees with $l$ for edges in $Y$, so $l'(e)=l'(\bar{e})$ for $e\in EY$ follows from the fact that $l$ is a legal labelling.	
	For an edge $e_i$, we know that $e_i\in E(\gamma_ix_i)$, so (\ref{l'def}) implies that
	\begin{equation}
		l'(e_i)=\theta(\gamma_i)l(\gamma_i^{-1}e_i)=f_il(\gamma_i^{-1}e_i).
	\end{equation}
On the other hand, the definition of $f_i$ says that $l_{\gamma_ix_i}\circ\gamma_i\circ l_{x_i}^{-1}=f_i$, so $l(e_i)=f_il(\gamma_i^{-1}e_i)=l'(e_i)$.
Furthermore, $\bar{e}_i\in E(y)$ for some $y\in VY$, so (\ref{l'def}) implies that $l'(\bar{e}_i)=l(\bar{e}_i)$.
Since $l$ is a legal labelling, we deduce that $l'(e_i)=l(e_i)=l(\bar{e}_i)=l'(\bar{e}_i)$.
Finally, for an edge of the form $he$ with $h\in\G$ and either $e\in EY$ or $e=e_i$, the $\theta$-equivariance of $l'$ implies that $l'(he)=\theta(h)l'(e)=\theta(h)l'(\bar{e})=l'(h\bar{e})$.
Thus $l'$ is a legal labelling.

Thirdly, we claim that $l'=(f_x)\circ l$ for some $(f_x)\in F^{VT}$.
Indeed, composing (\ref{l'def}) with $l_{hy}^{-1}$ yields $l'_{hy}\circ l_{hy}^{-1}=\theta(h)\circ l_y\circ h^{-1}\circ l_{hy}^{-1}$.
Since $h^{-1}\in\G<U^{(l)}(F)$, we know that $l_y\circ h^{-1}\circ l_{hy}^{-1}\in F$, hence $l'_{hy}\circ l_{hy}^{-1}\in F$.
The claim follows since $hy$ ranges over all vertices of $T$.
	
Since $l'$ is a legal labelling, we can form the lattice $\La^{l'}(F)$, and the $\theta$-equivariance of $l'$ implies that $\G<\La^{l'}(F)$ (see (\ref{fallx})). As $l,l'$ are both legal labellings, there is $g\in\Aut(T)$ with $l'=l\circ g$ -- see \cite[\S3.2]{BurgerMozes00b} or \cite[Lemma 4.3]{GarridoGlasnerTornier18}.
But we also have $l'=(f_x)\circ l$ for some $(f_x)\in F^{VT}$, so $g\in U^{(l)}(F)$ by Lemma \ref{lem:lcircg}.
Then $\G<\La^{l'}(F)=\La^{l\circ g}(F)=g^{-1}\La^l(F) g$ by Lemma \ref{lem:Laconj}, so $\La^l(F)$ contains a $U^{(l)}(F)$-conjugate of $\G$, as required.
\end{proof}

\bigskip
\section{Trees with fins}\label{sec:fins}

\begin{defn}(Trees with fins)\label{defn:graphfin}\\
	Let $X$ be a tree.
	Let $\Delta$ be a collection of embeddings of trees $\gamma:Y\to X$.
	A \emph{tree with fins} $\bfX$ is a CAT(0) square complex obtained by taking the mapping cylinder of  $$\bigcup_\Delta \gamma:\bigsqcup_\Delta Y\to X.$$
	As a space,
	$$\bfX=X\sqcup\bigsqcup_\Delta Y\times[0,1]/\sim,$$
	where $(y,0)\sim\gamma(y)$ for any $y\in Y$ and $(\gamma:Y\to X)\in\Delta$.
	The square complex structure comes from the subspaces $e\times[0,1]\subset\bfX$ for edges $e\in EY$ and $(\gamma:Y\to X)\in\Delta$.
	
	Note that there is a natural embedding $X\hookrightarrow{}\bfX$, and we will just write $X$ for the image of this embedding.
	Meanwhile, each subcomplex $Y\times\{1\}\subset\bfX$ is isomorphic to $Y$, and is referred to as a \emph{fin} of $\bfX$.
	For ease of notation we will always write $Y$ instead of $Y \times \{1\}$.
	
	A tree with fins $\bfX$ is \emph{locally} \emph{finite} if it is a locally finite square complex.
\end{defn}

\begin{defn}(Automorphism groups of trees with fins)\label{defn:autofin}\\
	Let $\bfX$ be a locally finite tree with fins.
We define the automorphism group $\Aut(\bfX)$ as the group of all cubical automorphisms of $\bfX$ that preserve the underlying tree $X\subset\bfX$ (and hence also preserve the collection of fins in $\bfX$).
Equipped with the compact-open topology, $\Aut(\bfX)$ is a locally compact group.
In analogy with the tree setting (Definition \ref{defn:unilattice}), we say that $\G<\Aut(\bfX)$ is a \emph{uniform lattice} if $\G$ acts properly and cocompactly on $\bfX$, and we say that $\bfX$ is \emph{uniform} if $\Aut(\bfX)$ contains a uniform lattice.
\end{defn}

\begin{remk}\label{remk:preserveX}
	For many trees with fins $\bfX$ (including all those we work with in this section) every cubical automorphism of $\bfX$ preserves the underlying tree $X\subset\bfX$ -- hence $\Aut(\bfX)$ can also be thought of as simply the group of all cubical automorphisms of $\bfX$.
	More precisely, if $X$ is a tree with no degree 1 vertices and if $\bfX$ has at least two fins, then every cubical automorphism of $\bfX$ preserves $X$. Indeed, in this case $\bfX$ is a CAT(0) square complex, and the collection of hyperplanes of $\bfX$ which bound a shallow halfspace (i.e. where the halfspace is contained in a bounded neighbourhood of the hyperplane) is precisely the collection of hyperplanes of the form $Y\times\{\tfrac{1}{2}\}\subset\bfX$, so this collection of hyperplanes must be preserved by all cubical automorphisms of $\bfX$; since $\bfX$ has at least two fins, it follows that all cubical automorphisms preserve $X$.
\end{remk}

\begin{remk}
	The definition of tree with fins given in \cite{Woodhouse21b} and \cite{ShepherdWoodhouse22} requires that each fin $Y$ is a bi-infinite line. We adopt the more general definition above because it is more convenient for the arguments in this section (and it is still a natural definition).
\end{remk}

The following proposition allows us to encode the information of any given universal group $U^{(l)}(F)<\Aut(T)$ from Section \ref{sec:universal} by attaching fins to the tree $T$, thus creating a tree with fins $\bsT$ with automorphism group isomorphic to $U^{(l)}(F)$.

\begin{prop}\label{prop:unitofins}
	Let $U^{(l)}(F)$ be the universal group of a tree $T$ with respect to a labelling $l$ and a group $F<S_n$, as in Definition \ref{defn:unigroup}.
	Then there is a subdivision $\sT$ of $T$, and a locally finite tree with fins $\bsT$ such that the natural map $\Aut(\bsT)\to\Aut(T)$ induces an isomorphism $\Aut(\bsT)\to U^{(l)}(F)$ of topological groups.
\end{prop}
\begin{proof}
Let $\sT$ be obtained from $T$ by subdividing each edge into $n$ edges. For a vertex $x\in VT$ let $x^*$ denote the corresponding vertex in $\sT$, and for an edge $e\in ET$ let $e^*$ denote the corresponding subtree of $\sT$ (i.e., the image of $e^*$ under the topological map $T\to\sT$).
For each $x\in VT$ and $f\in F$ we construct an embedding of trees $\gamma_{x,f}:Y_{x,f}\to \sT$ as follows (see also Figure \ref{fig:Yxf}):
\begin{itemize}
	\item The graph $Y_{x,f}$ is a finite tree obtained by taking $n$ paths of lengths $1,2,\dots,n$ respectively and identifying their start vertices. We write $y_{x,f}$ for this identified vertex, and we write $P_{i,x,f}\subset Y_{x,f}$ for the path of length $i$ that starts at $y_{x,f}$.
	\item The map $\gamma_{x,f}:Y_{x,f}\to \sT$ is the unique combinatorial embedding that maps $y_{x,f}$ to $x^*$, and maps each $P_{i,x,f}$ to an initial segment of $e^*$, where $e\in E(x)$ is the edge with $l(e)=f(i)$. 
\end{itemize}
The embeddings $\gamma_{x,f}$ then form the tree with fins $\bsT$ as in Definition \ref{defn:graphfin}.
Clearly $\bsT$ is locally finite.

Now let $\mathbf{g}\in\Aut(\bsT)$. 
This induces an automorphism of $\sT$ by restriction, and an automorphism $g\in\Aut(T)$ via the map $T\to\sT$.
We claim that $g\in U^{(l)}(F)$.
Let $x\in VT$.
The fin $Y_{x,1}$ (here $1$ denotes the identity of $F$) must map under $\mathbf{g}$ to a fin $Y_{gx,f}$ for some $f\in F$.
The map $\mathbf{g}:Y_{x,1}\to Y_{gx,f}$ must map $y_{x,1}$ to $y_{gx,f}$ and it must map each path $P_{i,x,1}$ to a path of the same length, namely the path $P_{i,gx,f}$.
Let $e_1\in E(x)$ with $l(e)=i$.
The path $P_{i,x,1}$ maps to $e_1^*$ under $\gamma_{x,1}$ and to $P_{i,gx,f}$ under $\mathbf{g}$.
And $P_{i,gx,f}$ maps to $e_2^*$ under $\gamma_{gx,f}$, with $e_2\in E(gx)$ the edge with $l(e_2)=f(i)$. Therefore $ge_1=e_2$.
This holds for all $i\in\{1,\dots,n\}$, therefore $l_{gx}\circ g\circ l_x^{-1}=f$ (notation from Definition \ref{defn:taulegal}).
This holds for all $x\in VT$, thus $g\in U^{(l)}(F)$.
This shows that the natural map $\Aut(\bsT)\to\Aut(T)$ has image contained in $U^{(l)}(F)$.

Next we show that $\Aut(\bsT)\to U^{(l)}(F)$ is surjective.
Let $g\in U^{(l)}(F)$.
Let $x\in VT$ and suppose that $l_{gx}\circ g\circ l_x^{-1}=f_x\in F$ ($f_x$ exists since $g\in U^{(l)}(F)$).
Let $f\in F$ and $i\in\{1,\dots,n\}$. The path $P_{i,x,f}\subset Y_{x,f}$ maps to $e_1^*$ under $\gamma_{x,f}$, where $e_1\in E(X)$ is the edge with $l(e_1)=f(i)$.
And the path $P_{i,gx,f_xf}\subset Y_{gx,f_xf}$ maps to $e_2^*$ under $\gamma_{gx,f_xf}$, where $e_2\in E(gx)$ is the edge with $l(e_2)=f_xf(i)$. But since $l_{gx}\circ g\circ l_x^{-1}=f_x$, we see that $ge_1=e_2$. This holds for all $i$, so the unique graph isomorphism $Y_{x,f}\to Y_{gx,f_xf}$ (which maps each path $P_{i,x,f}$ to the path $P_{i,gx,f_xf}$) fits into a commutative diagram:
\begin{equation}\label{finmap}
	\begin{tikzcd}[
		ar symbol/.style = {draw=none,"#1" description,sloped},
		isomorphic/.style = {ar symbol={\cong}},
		equals/.style = {ar symbol={=}},
		subset/.style = {ar symbol={\subset}}
		]
		Y_{x,f}\ar{d}{\gamma_{x,f}}\ar{r}&Y_{gx,f_xf}\ar{d}{\gamma_{gx,f_xf}}\\
		\sT\ar{d}&\sT\ar{d}\\
		T\ar{r}{g}&T
	\end{tikzcd}
\end{equation}
This holds for all $x,f$, therefore the automorphism $g\in\Aut(T)$ extends to an automorphism $\mathbf{g}\in\Aut(\bsT)$. (Note that for each $x\in VT$, $\mathbf{g}$ induces a bijection between the fins $\{Y_{x,f}\}_{f\in F}$ and $\{Y_{gx,f}\}_{f\in F}$ given by $Y_{x,f}\mapsto Y_{gx,f_xf}$, which is a genuine bijection because $F$ is a group.)
This shows  that $\Aut(\bsT)\to U^{(l)}(F)$ is surjective, as claimed.

Now we show that the map $\Aut(\bsT)\to U^{(l)}(F)$ has trivial kernel. 
In fact, we show the stronger statement that if $\mathbf{g}\in\Aut(\bsT)$ fixes a vertex $x\in VT$ as well as all the subtrees $e^*$ with $e\in E(x)$, then $\mathbf{g}$ fixes all the fins $Y_{x,f}$ with $f\in F$.
Indeed, for $f\in F$, the fin $Y_{x,f}$ must map to some fin $Y_{x,f'}$. For each $i$, the path $P_{i,x,f}$ must map to the path $P_{i,x,f'}$ under $\mathbf{g}$, and these paths must both map to the same $e_i^*$ under $\gamma_{x,f}$ and $\gamma_{x,f'}$, with $e_i\in E(x)$ (since $\mathbf{g}$ fixes all the $e^*$ with $e\in E(x)$ by hypothesis). But we also have $f(i)=l(e_i)=f'(i)$ by the definition of $\gamma_{x,f}$ and $\gamma_{x,f'}$, hence $f=f'$ and $\mathbf{g}$ fixes the fin $Y_{x,f}$.

Finally, we claim that the isomorphism $\Aut(\bsT)\to U^{(l)}(F)$ is an isomorphism of topological groups.
Let $x\in VT$.
By construction, for $k\geq1$, any $\mathbf{g}\in\Aut(\bsT)$ which fixes the $nk$-ball about $x$ in $\bsT$ induces an automorphism $g\in U^{(l)}(F)$ which fixes the $k$-ball about $x$ in $T$.
This shows that any sequence converging to the identity in $\Aut(\bsT)$ will map to a sequence that converges to the identity in $U^{(l)}(F)$.
Conversely, for $k\geq1$, if an automorphism $g\in U^{(l)}(F)$ fixes the $k$-ball about $x$ in $T$, then the claim from the previous paragraph implies that the extension $\mathbf{g}\in\Aut(\bsT)$ fixes the $(k-1)n$-ball about $x^*$ in $\bsT$.
This shows that any sequence converging to the identity in $U^{(l)}(F)$ will map to a sequence that converges to the identity in $\Aut(\bsT)$, thus proving the claim.
\end{proof}

\begin{figure}[H]
	\centering
	\begin{tikzpicture}[auto,node distance=2cm,
		thick,every node/.style={circle,draw,font=\small},
		every loop/.style={min distance=2cm},
		hull/.style={draw=none},
		]
		\tikzstyle{label}=[draw=none,font=\normalsize]
		\tikzstyle{sm}=[draw=blue,scale=.5,fill=blue]
		\tikzstyle{smb}=[scale=.5,fill]
		\node[sm] at (0,0) {};
		
		\draw[dashed] ([shift=(80:2)]0,0) arc (80:120:2);
		\draw[dashed] ([shift=(40:2)]0,0) arc (40:0:2);
	
		\path[blue] (0,0) edge  ++(180:1);
		\path[blue] (0,0) edge  ++(160:2);
		\path[blue] (0,0) edge  ++(140:3);
		\path[blue] (0,0) edge  ++(60:4);
		
		
		\node[sm] at ++(180:1){};
		\node[sm] at ++(160:1){};
		\node[sm] at ++(160:2){};
		\node[sm] at ++(140:1){};
		\node[sm] at ++(140:2){};
		\node[sm] at ++(140:3){};
		\node[sm] at ++(60:1){};
		\node[sm] at ++(60:2){};
		\node[sm] at ++(60:3){};
		\node[sm] at ++(60:4){};
		
		\node[label,blue,font=\huge] at (0,5){$Y_{x,f}$};
		\node[label,font=\large] at (4,.6){$\gamma_{x,f}$};
		\node[label,blue,font=\large] at (0,-.4){$y_{x,f}$};
		\node[label,blue,font=\large] at (2.5,3.7){$P_{i,x,f}$};
		\node[label,blue,font=\large] at (-1.7,0){$P_{1,x,f}$};
		\node[label,blue,font=\large] at (-2.5,1){$P_{2,x,f}$};
		\node[label,blue,font=\large] at (-2.6,2.3){$P_{3,x,f}$};
		
		\draw[draw=black,-triangle 90, ultra thick] (3,1) -- (5,1);
		
		\begin{scope}[shift={(10.5,0)}]
			\node[sm] at (0,0) {};
			
			\draw[dashed] ([shift=(80:2)]0,0) arc (80:120:2);
			\draw[dashed] ([shift=(40:2)]0,0) arc (40:0:2);
			
			\path (0,0) edge  ++(180:5);
			\path (0,0) edge  ++(160:5);
			\path (0,0) edge  ++(140:5);
			\path (0,0) edge  ++(60:5);
			
			\path[blue] (0,0) edge  ++(180:1);
			\path[blue] (0,0) edge  ++(160:2);
			\path[blue] (0,0) edge  ++(140:3);
			\path[blue] (0,0) edge  ++(60:4);
			
			\node[sm] at ++(180:1){};
			\node[smb] at ++(180:2){};
			\node[smb] at ++(180:3){};
			\node[smb] at ++(180:4){};
			\node[smb] at ++(180:5){};
			\node[sm] at ++(160:1){};
			\node[sm] at ++(160:2){};
			\node[smb] at ++(160:3){};
			\node[smb] at ++(160:4){};
			\node[smb] at ++(160:5){};
			\node[sm] at ++(140:1){};
			\node[sm] at ++(140:2){};
			\node[sm] at ++(140:3){};
			\node[smb] at ++(140:4){};
			\node[smb] at ++(140:5){};
			\node[sm] at ++(60:1){};
			\node[sm] at ++(60:2){};
			\node[sm] at ++(60:3){};
			\node[sm] at ++(60:4){};
			\node[smb] at ++(60:5){};
			
			\node[label,font=\huge] at (0,5){$\sT$};
			\node[label,font=\large] at (0,-.4){$x^*$};
			\node[label,font=\large] at (2.6,4.7){$e^*$};
		\end{scope}
		
	\end{tikzpicture}
	\caption{\small The map $\gamma_{x,f}:Y_{x,f}\to \sT$. The subtree $e^*\subset\sT$ indicated is associated with $e\in E(x)$ such that $l(e)=f(i)$.}\label{fig:Yxf}
\end{figure}
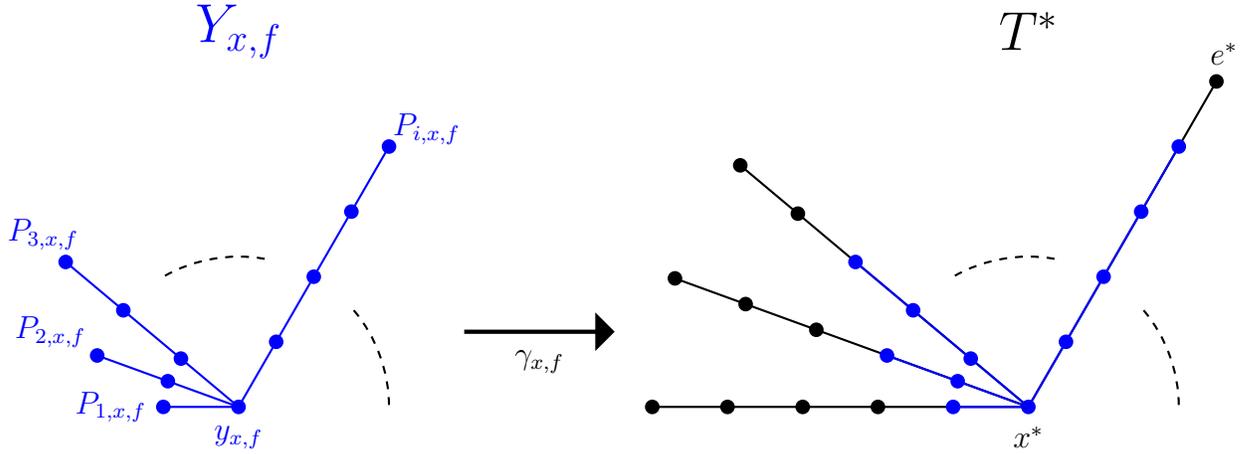

We can now deduce Theorems \ref{thm:fintree} and \ref{thm:finquotients} using the results from Section \ref{sec:universal}.

\theoremstyle{plain}
\newtheorem*{thm:fintree}{Theorem \ref{thm:fintree}}
\begin{thm:fintree}
	There exists a locally finite tree with fins $\bfX$ and free uniform lattices $\G,\G'<\Aut(\bfX)$ such that no uniform lattice $\La<\Aut(\bfX)$ contains conjugates of both $\G$ and $\G'$.
\end{thm:fintree}
\begin{proof}
	Theorem \ref{thm:univ} provides us with a $\tau$-legal labelling $l$ on a tree $T$, a universal group $U^{(l)}(F)$, and free uniform lattices $\G,\G'<U^{(l)}(F)$ such that no uniform lattice $\La<U^{(l)}(F)$ contains $U^{(l)}(F)$-conjugates of both $\G$ and $\G'$.
	By Proposition \ref{prop:unitofins}, we obtain a locally finite tree with fins $\bsT$ and an isomorphism of topological groups $\Aut(\bsT)\to U^{(l)}(F)$.
	The theorem follows by mapping the free uniform lattices $\G,\G'<U^{(l)}(F)$ over to $\Aut(\bsT)$ and taking $\bfX=\bsT$.
	Note that the images of $\G,\G'$ in $\Aut(\bsT)$ will act freely on $\bsT$, because if $g\in\G$ (or $g\in\G'$) corresponds to $\mathbf{g}\in\Aut(\bsT)$ that fixes a point $x\in\bsT$, then $\mathbf{g}$ will also fix the point $r(x)\in\sT$, where $r:\bsT\to\sT$ is the natural retraction map, and $g$ will fix the corresponding point in $T$.
\end{proof}

\theoremstyle{plain}
\newtheorem*{thm:finquotients}{Theorem \ref{thm:finquotients}}
\begin{thm:finquotients}
There exists a locally finite uniform tree with fins $\bfX$ with the following properties:
\begin{enumerate}
	\item\label{it:diffquotients} There is no uniform lattice $\G<\Aut(\bfX)$ such that the map $\G\backslash\bfX\to\Aut(\bfX)\backslash\bfX$ is an isomorphism.
	\item\label{it:diffperms} Let $X\subset\bfX$ be the underlying tree. 
	For $x\in VX$, let $E_X(x)$ denote the collection of edges of $X$ that are incident to $x$.
	There is no uniform lattice $\G<\Aut(\bfX)$ such that, for every $x\in VX$, every permutation of $E_X(x)$ induced by an automorphism in $\Aut(\bfX)_x$ is also induced by an automorphism in $\G_x$.
\end{enumerate}
\end{thm:finquotients}
\begin{proof}
	Take the $\tau$-legal labelling on a tree $T$ and the universal group $U^{(l)}(F)$ from Example \ref{exmp:fewerorbits}.
	Apply Proposition \ref{prop:unitofins} to get a locally finite tree with fins $\bsT$ and an isomorphism of topological groups $\Aut(\bsT)\to U^{(l)}(F)$ induced by the natural map $\Aut(\bsT)\to\Aut(T)$.
	Take $\bfX=\bsT$.
	As explained in Example \ref{exmp:fewerorbits}, every uniform lattice in $U^{(l)}(F)$ has strictly more orbits of geometric edges in $T$ than $U^{(l)}(F)$ itself. This implies that $\bfX$ satisfies part \ref{it:diffquotients} of the theorem.
	Furthermore, it is explained in Example \ref{exmp:fewerorbits} that there is no uniform lattice $\G< U^{(l)}(F)$ with the property that the map $\sigma_{x,l}:\G_x\to F$ is a surjection for every $x\in VT$.
	The map $\sigma_{x,l}:\G_x\to F$ is a restriction of the map $\sigma_{x,l}:U^{(l)}(F)_x\to F$, which describes how $U^{(l)}(F)_x$ permutes the edges in $E(x)$. Moreover, $\sigma_{x,l}: U^{(l)}(F)_x\to F$ is surjective by Proposition \ref{prop:local=F}\ref{it:local=F}.
	Therefore, there is no uniform lattice $\G< U^{(l)}(F)$ such that, for every $x\in VT$, every permutation of $E(x)$ induced by an automorphism in $U^{(l)}(F)_x$ is also induced by an automorphism in $\G_x$.
	This implies that $\bfX$ satisfies part \ref{it:diffperms} of the theorem.	
\end{proof}

\bibliographystyle{alpha}
\bibliography{Ref}

\end{document}